\documentclass[12pt]{article}
\usepackage{graphicx}
\usepackage{algorithm}
\usepackage{stackengine}
\usepackage{caption}
\usepackage{cite}
\usepackage{color}
\usepackage{xcolor}
\usepackage{amsmath,amsthm,amssymb,amsfonts,bm}
\usepackage{textcomp}
\usepackage{multirow}
\usepackage{subcaption}
\usepackage{makecell}
\usepackage{algorithm,algorithmic}

\makeatletter
\newcommand*{\rom}[1]{\expandafter\@slowromancap\romannumeral #1@}
\usepackage{appendix}
\usepackage{authblk}
\usepackage{booktabs} 
\usepackage{array}    
\usepackage{multirow} 
\usepackage{tabularx}

\newtheorem{lemma}{Lemma}
\newtheorem{theorem}{Theorem}

\newtheorem{remark}{Remark}

\newtheorem{definition}{Definition}

\newtheorem{question}{Question}

\makeatletter
\newenvironment{breakablealgorithm}
  {
   \begin{center}
     \refstepcounter{algorithm}
     \hrule height.8pt depth0pt \kern2pt
     \renewcommand{\caption}[2][\relax]{
       {\raggedright\textbf{\ALG@name~\thealgorithm} ##2\par}%
       \ifx\relax##1\relax 
         \addcontentsline{loa}{algorithm}{\protect\numberline{\thealgorithm}##2}%
       \else 
         \addcontentsline{loa}{algorithm}{\protect\numberline{\thealgorithm}##1}%
       \fi
       \kern2pt\hrule\kern2pt
     }
  }{
     \kern2pt\hrule\relax
   \end{center}
  }
\makeatother

\title{Deterministic Algorithms to Solve the $(n,k)$-Complete Hidden Subset Sum Problem}
\author{Lixia Luo\thanks{Department of Mathematics and Statistics, Hunan University of Science and Technology, China (luolixia@hnust.edu.cn)}, Changheng Li \thanks{EEMCS department, Delft University of Technology, the Netherlands (L.Li-7@tudelft.nl)}, Qiongxiu Li \thanks{Department of Electronic System, Aalborg University, Denmark (qili@es.aau.dk)}}

\begin{document}
\date{}
\maketitle


%
\begin{abstract}
The Hidden Subset Sum Problem (HSSP) is a significant NP-complete problem in number theory and combinatorics, with applications in cryptography and AI privacy. For the $(n,k)$-complete HSSP, where a target multiset must be recovered from its all $k$-subset sums, existing algorithms face limitations due to high complexity or intractability. This paper proposes two deterministic algorithms: a brute-force approach, and a novel method leveraging symmetric polynomials and Vieta’s formulas with $O\left(\sum_{u=1}^n p(u,\leq k)^3+\binom{n}{k}n\right)$ complexity, where $ p(u,\leq k)$ counts the number of partitions of a positive integer $u$ into at most $k$ parts. The latter constructs an $n$-th degree polynomial via Vieta’s formulas, whose roots correspond to the hidden multiset elements. Additionally, the discussion about the homogeneous symmetric polynomial rings is of independent interest.

\end{abstract}
\section{Introduction}

The \textbf{Hidden Subset Sum Problem (HSSP)} is a classic and intriguing problem in additive number theory, additive combinatorics, and combinatorial number theory. Its broad practical applications span various fields such as cryptography  \cite{boyko1998speeding}, medical  radiation oncology \cite{collins2007nonnegative}, AI privacy \cite{li2024privacy,li2024perfect}, and Wi-Fi channel allocation \cite{busson2024channel}. Notably, HSSP has been independently studied under various names and contexts, including as the  \textbf{multiset recovery problem} \cite{fomin2019multiset} and, in some cases, as a special instance of the \textbf{generating set problem} \cite{collins2007nonnegative}. 

\vspace{1em}
\noindent\textbf{Definitions of HSSP} Let $G$ be an Abelian group. The general definition of HSSP defined over $G$ is as follows: 

\textit{Given a multiset $B$ of $G$ consisting of $m$ elements, the goal is to find a finite multiset $X$ of cardinality $n(n \leq m)$ such that every element of $B$ is a subset sum of elements from the original multiset $X$.} 

Unlike ordinary sets, multisets allow elements to have multiplicities, meaning the same element can appear multiple times within the set \cite{blizard1989multiset}. In contrast to the \textbf{ Subset Sum Problem (SSP)}, where given a finite multiset $X$ and a value $b$, the goal is to find a subset of $X$ such that the sum of its elements equals $b$, HSSP hides the multiset $X$ and typically only provides several subset sums. In particular, if the multiset consists of all $k$-subset sums of $G$ (i.e., sums formed from subsets of $G$ containing exactly $k$ elements), it is referred to as an \textbf{$(n,k)$-complete HSSP}, which is the case we focus on in this article. Additionally, HSSP can also be formulated as follows:

\textit{Given an $m$-dimensional vector $\mathbf{B}=(b_1,b_2,\dots,b_m) \in G^m$, find a matrix $\mathbf{A}\in \{0,1\}^{m \times n}$ and an $n$-dimensional vector $\mathbf{X}=(x_1,x_2,\dots,x_n)\in G^n$ such that $\mathbf{A}\mathbf{X} = \mathbf{B}$.} 

Moreover, the underlying group structure can be generalized to a ring or field framework while ensuring that the problem remains well-defined, thereby leading to different variations of HSSP. 

\vspace{1em}
\noindent\textbf{Related works of $(n,k)$-complete HSSP} 
The study of $(n,k)$-complete HSSP has a rich history dating back to 1957 when Moser first published a solution for the $(5,2)$ case and discussed potential multiple solutions for the $(4,2)$ variant \cite{moser1957e1248}. This early work laid the groundwork for subsequent investigations into the multi-solution nature of these problems. In 2019, Fomin summarized and formalized these types of inquiries under the name multiset recovery problems in a review article \cite{fomin2019multiset}, dedicated for memorizing Izhboldin. This class of problems also finds its place among the unsolved challenges in additive number theory, as noted by Guy in his comprehensive compendium, \textit{Unsolved Problems in Number Theory} \cite{guy2004unsolved}.

In 1958, Selfridge and Straus \cite{selfridge1958determination} applied symmetric function theory to prove that the solution to the $(n,2)$-complete HSSP is unique if and only if $n$ is not a power of $2$. For the general case where $k \neq 2$, they provided a sufficient condition for the uniqueness of the solution through a family of Diophantine equations of the form $P(n, k) = 0,$
where $ P(n,k) = \sum_{i=1}^{k} (-1)^{i-1} i^{u-1} \binom{n}{k-i}$ is referred to as the \textbf{Moser polynomials} \cite{fomin2019multiset}. Here, $u$ ranges over $1, 2, \dots, n$. If $(n,k)$ does not satisfy any of these equations, the $(n,k)$-complete HSSP has a unique solution. However, if $(n,k)$ satisfies at least one of these equations, it is unclear whether multiple solutions exist.

In 1962, Gordon, Fraenkel, and Straus \cite{gordon1962determination} proved that for $(n,k)$-complete HSSP defined on torsion-free Abelian groups, there are only finitely many values of $n$ for which multiple solutions exist, given a fixed $k$. They also provided an upper bound on the number of solutions in some cases. These parameter pairs $(n,k)$ that yield multiple solutions are referred to as as \textbf{singular pairs}. Subsequent research by Ewell, Fomin, Izhboldin, and others further investigated the determination of singular pairs and the upper bounds on the number of corresponding solutions, leading to several open questions \cite{ewell1968determination,fomin1995collections,boman1996examples,isomurodov2017set}. Recently, Ciprietti and Glaudo examined the uniqueness of the solution for HSSP when the multiset is a subset sum of all subsets of a given set \cite{ciprietti2023determination}.

\vspace{1em}
\noindent\textbf{Our Contributions}
While the multi-solution nature of the $(n,k)$-complete HSSP has been extensively studied using symmetric function theory, no algorithms have been proposed to solve it exclusively. In fact, deterministic algorithms for solving the $(n,k)$-complete HSSP are crucial for experimentally investigating its multi-solution properties. However, to the best of our knowledge, no deterministic algorithms have been proposed for solving the $(n,k)$-complete HSSP.

 This paper focuses on developing deterministic algorithms for solving the $(n,k)$-complete HSSP. In this article, let $X_{n,k}$ denote the multiset consisting of all $k$-subset sums.  Without loss of generality, we assume $k\leq n/2$ as the sum of all $x_i$'s can be determined from  $X_{n,k}$. Consequently, all $\left(n-k\right)$-subset sums can be obtained by subtracting each $k$-subset sum from the total sum of $x_i$'s. Let $X,X_{n,k}$ form the vectors $\mathbf{X},\mathbf{X}_{n,k}$, respectively. Their relationship can be formulated as $\mathbf{AX}=\mathbf{X}_{n,k}$, where $\mathbf{A}$ is a $\binom{n}{k} \times n$ binary matrix, with each row containing exactly $k$ ones. For simplicity, we consider the $(n,k)$-complete HSSP defined over the field of real numbers $\mathbb{R}$.

Two types of deterministic algorithms are proposed in this work. The first is based on \textbf{brute-force search}. By leveraging \textbf{the ordering relationships of subset sums}, we have reduced the complexity of the direct brute-force search algorithm. The result is formalized in the following theorem:
 
\begin{theorem}\label{thm:bru}
There is a deterministic algorithm to solve the $(n,k)$-complete HSSP defined over $\mathbb{R}$ via brute-force search, with a time complexity of
\begin{equation*}
    O\left(\binom{n}{k}\left(\log \binom{n}{k}+\prod_{i=1}^{k-1}\left(\binom{n-k+i}{i}-i\right)\right)\right).
\end{equation*}
\end{theorem}

The second algorithm is based on \textbf{symmetric function theory} and \textbf{Vieta's formulas}. First, the elementary symmetric polynomials of all elements in $X$ are computed from all $k$-subset sums of $X$. Then, $X$ is recovered by finding the roots of a univariate polynomial equation constructed using Vieta's formulas. Moreover, \textbf{Newton's identities} are used to reduce computation.
The result is summarized as follows:
 
\begin{theorem}\label{thm:vie}
For an $(n,k)$-complete HSSP defined over $\mathbb{R}$, if $\sum_{i=1}^{u}(-1)^{i-1}(i -1)!S(u,i)\binom{n-i}{k-i}\ne 0$ for $u=1,2,\dots,n$, there is a deterministic algorithm, called \textbf{symmetric subset sum recovery algorithm}, to solve it with a time complexity of
$$O\left(\sum_{u=1}^n p(u,\leq k) ^3+\binom{n}{k}n\right),$$ where $p(u,\leq k) $ denotes the partition function,  which counts the number of partitions of a positive integer $u$ into at most $k$ parts.
\end{theorem}

\vspace{1em}
\noindent\textbf{Related works of HSSP} 
 HSSP has appeared under various names across different research areas. Besides being referred to as the multiset recovery problem in studies on the multi-solution nature of $(n,k)$-complete HSSP, it is also a special instance of the \textbf{generating set problem} \cite{collins2007nonnegative} in combinatorial number theory. The generating set problem has been studied in various contexts, with one notable approach by Collins et al. proving its NP-completeness \cite{collins2007nonnegative}. Some greedy heuristic algorithms have been proposed to address it \cite{collins2007nonnegative, lozano2016genetic}. More recent works have explored connections between HSSP and graph theory, including Busson et al.'s transformation of the Wi-Fi access point channel allocation problem into a 1-extendable partition problem \cite{busson2024channel}, which is equivalent to the generating set problem.

In cryptography, the HSSP over the ring of residues modulo $M$ has been applied to public-key cryptosystems, improving the efficiency of key generation through the BPV generator \cite{boyko1998speeding}. Subsequently, Nguyen and Stern (1999) proposed the NS algorithm, based on orthogonal lattice attacks, to solve these instances of HSSP \cite{nguyen1999hardness}. Gini and Coron followed their work and enhanced the NS algorithm, obtaining a polynomial-time probabilistic version \cite{coron2020polynomial, coron2021provably, gini2022hardness}. However, these methods rely on favorable properties of the matrix $\mathbf{A}$ and vector $\mathbf{B}$, making them less effective for solving the $(n,k)$-complete HSSP.

Recent studies in AI privacy have linked HSSP to federated learning \cite{li2024privacy, li2024perfect}. While techniques like differential privacy, homomorphic encryption, and secure multi-party computation protect intermediate data \cite{yin2021comprehensive}, few address theoretical privacy analysis from the shared updates. For example, \cite{li2024privacy} mapped the input data of the federated K-means clustering algorithm to a hidden multiset in HSSP, formulating intermediate results as subset sums. It transforms the input recovery attack into an instance of HSSP, connecting AI privacy with HSSP.

Unlike the aforementioned heuristic and probabilistic approaches, our work focuses on developing deterministic algorithms for solving the $(n,k)$-complete HSSP, offering a more structured and reliable solution framework.

\vspace{1em}
\noindent\textbf{Organization}
The structure of the paper is as follows. Section \ref{sec:bac} provides some necessary background. In Section \ref{sec:bru}, we present an optimized version of the brute-force search algorithm for solving the $(n,k)$-complete HSSP. Section \ref{sec:vie} discusses the proposed algorithm based on symmetric polynomials theory and Vieta's formulas in detail. Conclusion and future work are outlined in Section \ref{sec:fur}.

\section{Background}\label{sec:bac}

\noindent\textbf{Symmetric polynomials and Vieta's formulas}

\begin{definition}\cite{lang2012algebra}
    Let $x_1,x_2,\dots,x_n$ be algebraically independent elements over $\mathbb{R}$ and $S_n$ be the symmetric group on $n$ letters. A polynomial $f(x_1,x_2,\dots,x_n)$ in $\mathbb{R}[x_1,\dots,x_n]$ is called \textbf{symmetric} if 
$$f(x_{\sigma(1)},x_{\sigma(2)},\dots,x_{\sigma(n)})=f(x_1,x_2,\dots,x_n)$$
holds for all $\sigma \in S_n$.
\end{definition}

Consider the polynomial $f(x):=\prod_{i=1}^n(x-x_i)=x^n+\sum_{i=n-1}^{0}(-1)^{n-i}e_{n-i}x^i$. By Vieta's formulas\cite{lang2012algebra}, the coefficients $e_i$'s satisfy:
$$\begin{aligned}
e_1&:=x_1+x_2+\cdots+x_n, \\
e_2&:=x_1x_2+x_1x_3+\cdots+x_{n-1}x_{n}, \\
 e_3&:=x_1x_2x_3+x_1x_2x_4+\cdots+x_{n-2}x_{n-1}x_n,\\
 &\cdots\\
e_n&:=x_1x_2\cdots x_n.
\end{aligned}
$$
This implies that the coefficients $e_i$'s can be expressed as symmetric polynomials in  $x_1,x_2,\dots,x_n$, known as elementary symmetric polynomials.

\vspace{1em}
\noindent\textbf{Partitions}
\begin{definition}\cite{andrews2004integer}
    A \textbf{partition} of positive integer $u$, also called an integer partition, is a way of expressing $u$ as a sum of positive integers. These positive integers are referred to as the parts of the partition. 
\end{definition}
Throughout this work, partitions will be represented as arrays of their parts, listed in ascending order. The number of parts in a partition $p$ will be referred to as its length, denoted by $n(p)$.

\begin{definition}\cite{andrews2004integer}
    The \textbf{partition function} $p(u)$ is defined as a function to count the number of all partitions of a positive integer $u$.
\end{definition}

\begin{definition}
    Define $p(u,\leq k)$ to be the number of all partitions of a positive integer $u$ into at most $k$ parts.  
\end{definition}

\vspace{1em}
\noindent\textbf{Stirling partition number}
\begin{definition}\cite{stanley2011enumerative}
   A \textbf{Stirling number of the second kind} or \textbf{Stirling partition number} $S(u,i)$ is defined to be the number of ways to partition a set of $n$ objects into $k$ non-empty subsets.
\end{definition}
The Stirling partition number can be calculated using the following explicit formula:
$$S(u,i)=\frac{1}{i!}\sum_{j=0}^i(-1)^{i-j}
\binom{i}{j}j^u .$$
Here are some examples of Stirling partition numbers:
$$S(u,u)=1,S(u,1)=1,S(u,2)=2^{u-1}-1,S(u,3)=\frac{1}{2}(3^{u-1}-2^u+1)$$
$$S(u,4)=\frac{1}{3!}(4^{u-1}-3^u+3\cdot2^{u-1}-1), S(u,5)=\frac{1}{4!}(5^{u-1} - 4\cdot4^{u-1} + 6\cdot3^{u-1} - 4\cdot2^{u-1} + 1). $$
Notice that $S(u,i)=0$ for any $u<i$.

\section{Brute-force search}\label{sec:bru}
Brute-force search, also known as exhaustive search, is a general problem-solving technique that systematically examines all possible candidates. For the 
$(n,k)$-complete HSSP, we know that each sum is the sum of $k$ elements from $X$, but we do not know which $k$ elements contribute to each sum. If we simply traverse all possible subsets of $X$ and check the subset sums, such a brute-force approach can be extremely time-consuming.

Hence, we might as well fix an invertible binary matrix $\mathbf{A}_n$ whose rows each contain exactly $k$ ones, and select $n$ possible subset sums from the total $\binom{n}{k}$ subset sums to identify the candidates of $X$. There are totally $\binom{\binom{n}{k}}{n}$ possible combinations. Afterward, check the unselected subset sums to approve or reject the candidates of $X$.  It is known that the time complexity of solving the system of linear equations is $O(n^3)$ \cite{trefethen1997numerical} and the complexity of checking the rest sums is $O(\binom{n}{k}-n)$. Thus, the total time complexity is $O(\binom{\binom{n}{k}}{n}(n^3+\binom{n}{k}-n))$.

\vspace{1em}
\noindent\textit{Proof of Theorem \ref{thm:bru}}
Since $\mathbb{R}$ is an ordered field, the ordering relations can be used to reduce the number of possible traversals. The steps of the optimal brute-force search algorithm can be described as follows.
\begin{enumerate}
    \item \textbf{Initialization and sorting:} 
    \begin{itemize}
        \item 
    Assume $x_1\leq x_2 \leq \cdots \leq x_n$. Define the square matrix $\mathbf{A}_{k+1}$ of order $k+1$, where all elements are  $1$, except for the anti-diagonal elements, which are 0.
        \item
     Sort all subset sums such that $b_1\leq b_2\leq \cdots \leq b_m.$
    \end{itemize}
    \item \textbf{Constructing matrix equation:} 
    
    Consider  $$\mathbf{A}_{k+1}\mathbf{X}_{k+1}=\mathbf{B}_{k+1},$$ where $\mathbf{B}_{k+1}=(b'_1,b'_2,\dots,b'_{k+1})$. We have $b'_1=b_1$ and $b'_2=b_2$ since $\mathbf{X}_{k+1}$ is in an ascending order. For the remaining elements, we have $$b'_3\in \{b_3,b_4,\dots,b_{(n-k+2)}\},b'_3\leq b'_4 \in \{b_4,b_5,\dots,b_{\binom{n-k+2}{2}+1}\},\dots,$$  $$ b'_{i-1} \leq b'_i\in \{b_i,\dots,b_{\binom{n-(k-i+2)}{i-2}+1}\},\dots,b'_{k}\leq b'_{k+1}\in\{b_{k+1},\dots,b_{\binom{n-1}{k-1}+1}\}.$$ Hence, for $i\geq 3$, we have $\binom{n-(k-i+2)}{i-2}-i+2$  choices for $b'_{i}$ in the worst case.

    \item \textbf{Traversing and solving:} 
    
    Traverse all possible $\mathbf{B}_{k+1}$ in step 2 and solve $\mathbf{A}_{k+1}\mathbf{X}_{k+1}=\mathbf{B}_{k+1}$ to get $\mathbf{X}_{k+1}$. 
    
    Now, scratching out the $b'_i$ for $i=1,2,\dots,k+1$, the smallest subset sum $b'_{k+2}$ of $\left\{b_1,b_2,\dots,b_{m}\right\}\setminus\left\{b'_1,b'_2,\cdots,b'_{k+1}\right\}$  is equal to $x_1+x_2+\cdots+x_{k-1}+x_{k+2}$, then we could get $x_{k+2}$. Generate subset sums using $x_1,x_2,\dots,x_{k+2}$. If these sums belong to  $\left\{b_1,b_2,\dots,b_{m}\right\}\setminus\left\{b'_1,b'_2,\cdots,b'_{k+1}\right\}$, scratch them out; otherwise, reject these $x_i$'s and turn to next possible $\mathbf{B}_{k+1}$. Again, the smallest subset sum of residual $b_i$'s is equal to $x_1+x_2+\cdots+x_{k-1}+x_{k+3}$. Repeat this process until all $x_i$'s are obtained.
\end{enumerate}

In order to  solve the matrix equation $\mathbf{A}_{k+1}\mathbf{X}_{k+1}=\mathbf{B}_{k+1}$ with reduced complexity, we show that  the inverse matrix of $\mathbf{A}_{k+1}$ could be calculated directly.

\begin{lemma}
    For the matrix $\mathbf{A}_{k+1}=(a_{ij})$ where $$a_{ij}=1-\delta_{i,k+2-j}=\begin{cases}
        1 & i+j\ne k+2\\
        0 & i+j=k+2
    \end{cases},$$ the inverse of $\mathbf{A}_{k+1}$ is 
    $\mathbf{A}_{k+1}^{-1}=(b_{ij})$ where $$b_{ij}=\begin{cases}
        \frac{1}{k} & i+j\ne k+2\\
        -\frac{k-1}{k} & i+j=k+2
    \end{cases}.$$
\end{lemma}
\begin{proof}
   It can be checked directly.
\end{proof}

The algorithm could be also given as follows.

\begin{breakablealgorithm}
    \renewcommand{\algorithmicrequire}{\textbf{Input:}}
	\renewcommand{\algorithmicensure}{\textbf{Output:}}
	\caption{Brute-force search}
    \label{alg:bru}
    \begin{algorithmic}[1] 
        \REQUIRE  Positive integers $k,n$ satisfying $2\leq k\leq n/2$ and the multiset $X_{n,k}=\{b_1,b_2,\dots,b_m\}$ where $m=\binom{n}{k}$; 
	    \ENSURE The hidden multiset $X=\{x_1,x_2,\dots,x_n\}$; 
        
        \STATE Sort the subset sums such that $b_1\leq b_2\leq \cdots \leq b_m$,  and construct the inverse matrix $\mathbf{A}_{k+1}^{-1}$ where $\mathbf{A}_{k+1}^{-1}[i,j]=\begin{cases}
             \frac{1}{k} & i+j\ne k+2\\
        -\frac{k-1}{k} & i+j=k+2
        \end{cases}$;

        \STATE Initialize an empty column vector $\mathbf{B}_{k+1}=[]$  of length $k+1$ and set $\mathbf{B}_{k+1}[1]= b_1$, $\mathbf{B}_{k+1}[2]=b_2$;
         \STATE Set $c=[1,2,\dots,k+1]$; $s=0$; 
         \WHILE {$s==0$}
           \STATE Set $\mathbf{C}=(b_3,b_4,\dots,b_m);$
           \FORALL {$i=3,4,\dots,k+1$}
               \STATE Set $\mathbf{B}_{k+1}[i]=b_{c[i]}$ and remove $b_{c[i]}$ from  $\mathbf{C}$;
          \ENDFOR
           \STATE Compute $\mathbf{X}_{k+1}=\mathbf{A}_{k+1}^{-1}\mathbf{B}_{k+1}$ and determine all $i$-subset sums of the elements of $\mathbf{X}_{k+1}$, denoted by the array $\mathbf{S}_i$, for all $i=1,2,\dots,k-1$, where $\mathbf{S}_{k-1}[1]=\mathbf{X}_{k+1}
          [1]+\mathbf{X}_{k+1}[2]+\cdots+\mathbf{X}_{k+1}[k-1];$ 
           \FORALL{$j=k+2,k+3,\dots,n$}
             \STATE Let $x=\mathbf{C}[1]-\mathbf{S}_{k-1}[1]$;
               \FORALL{$l=1,2,\dots,\binom{j-1}{k-1}$}
                \IF{$\mathbf{S}_{k-1}[l]+x$ is not in $\mathbf{C}$}
                 \STATE Break;
                 \ELSIF{$j==n$  and $\mathbf{C}==[]$}
                 \RETURN{$\mathbf{S}_{1}$ as a multiset.} // The elements of $\mathbf{S}_{1}$ are exactly $X$'s. 
                \ELSE 
                 \STATE Remove $\mathbf{S}_{k-1}[l]+x$ from $\mathbf{C}$;
               \ENDIF
             \ENDFOR
              \STATE Let $\mathbf{S}_1=[op(\mathbf{S}_1),x]$ and  $\mathbf{S}_{i}=[op(\mathbf{S}_{i}),op(\mathbf{S}_{i-1}+x)]$ for all $i=2,3,\dots,k-1$; // $op(\mathbf{S}_1)$ means opening the array $\mathbf{S}_1$.
           \ENDFOR
               \FORALL{$i=k+1,k,\dots,4$} 
                     \IF{$c[i]<\binom{n-(k-i+2)}{i-2}+1$}
                     \STATE Let $c[i]=c[i]+1$; break;
                    \ELSIF{$c[i]==\binom{n-(k-i+2)}{i-2}+1$ and $c[i-1]<\binom{n-(k-i+1)}{i-3}+1$}
                     \STATE Let $c[i-1]=c[i-1]+1$; 
                     \FORALL{$j=i,i+1,\dots,k+1$}
                        \STATE Let $c[j]=c[j-1]+1$;
                      \ENDFOR
                      \STATE Break;
                     \ENDIF
            \ENDFOR 
             
             \IF{$i==3$ and $c[3]>n-k$} 
                       \STATE Let $s=2$;
             \ENDIF
        \ENDWHILE
        \IF{$s==2$}
          \RETURN{No solution!}
        \ENDIF
\end{algorithmic}
\end{breakablealgorithm}

\noindent\textbf{Complexity analysis}
Using the heapsort algorithm, the time complexity for sorting the subset sums is  
$ O\left(\binom{n}{k} \log \binom{n}{k}\right) $\cite{leiserson1994introduction}. In the worst-case scenario, the number of iterations is bounded by  
$ O\left(\prod_{i=3}^{k+1} \left(\binom{n-(k-i+2)}{i-2} - i + 2 \right)\right)$,  and the time complexity of each iteration is  
$ O\left((k+1)^2 + \binom{n}{k} - k - 1\right). $  
Therefore, the total time complexity is:  
$$  
O\left(\binom{n}{k} \log \binom{n}{k} + \left(\binom{n}{k} + k^2 + k\right) \prod_{i=1}^{k-1} \left(\binom{n-k+i}{i} - i\right)\right), 
$$  
which, by ignoring lower-order terms, is asymptotically equivalent to $$
    O\left(\binom{n}{k}\left(\log \binom{n}{k}+\prod_{i=1}^{k-1}\left(\binom{n-k+i}{i}-i\right)\right)\right).$$ 
This completes the proof of Theorem \ref{thm:bru}. $\hfill\qedsymbol$

If Algorithm \ref{alg:bru} is allowed to continue exploring all possibilities instead of terminating upon finding the first solution, it can be modified to output all solutions.

\section{The approach based on Vieta's formulas}\label{sec:vie}
\subsection{Some notations and the preliminary idea }
In this section, for the hidden multiset $X=\{x_1,x_2,\dots,x_n\}$, let $P_j=\sum_{i=1}^{n}x_i^j$ denote the $j$-th power sum of all elements of $X$. If we take $\{x_1,x_2,\dots,x_n\}$ as $n$ variates, $P_j$'s are symmetric polynomial of $\{x_1,x_2,\dots,x_n\}$. In the following, we do not distinguish between treating 
$\{x_1,x_2,\dots,x_n\}$ as a multiset or as variables. Let $$P_{(j_1,j_2,\dots,j_{\ell})}=\sum_{1\leq i_1\leq n}x_{i_1}^{j_1}\left(\sum_{1\leq i_2\leq n,i_2\ne i_1}x_{i_2}^{j_2}\left(\cdots \sum_{1\leq i_{\ell}\leq n,i_{\ell}\ne i_1,\dots,i_{\ell-1}}x_{i_{\ell}}^{j_{\ell}} \right)\right)$$ be the other symmetric polynomials with exponent form $(j_1,j_2,\dots,j_{\ell})$, where $1\leq j_1\leq j_2\leq \cdots \leq j_{\ell} \leq n.$ $(j_1,j_2,\dots,j_{\ell})$ could be taken as a partition of $j_1+j_2+\cdots+j_{\ell}$ with parts in ascending order. All the monomials in
$P_{(j_1,j_2,\dots,j_{\ell})}$ have the same coefficients, denoted by $\mu(j_1,j_2,\dots,j_{\ell}).$ In fact, let $\#(m)$ denote the multiplicity of any positive integer $m$ in $(j_1,j_2,\dots,j_{\ell})$, then $$\mu(j_1,\dots,j_{\ell})=\prod_{j_1\leq m\leq j_{\ell}} \#(m)!.$$ In addition, denote the $u$-th power sum of the multiset $X_{n,k}$ consisting of all $k$-subset sums of $X$ by
$$ S_{u}=\sum_{1\leq i_1<i_2<\cdots<i_k\leq n}(x_{i_1}+x_{i_2}+\cdots+x_{i_k})^u .$$ 

\noindent\textbf{The preliminary idea of this method}

Since $S_{u}$ can be computed for any $u$ using $X_{n,k}$, if the values of other symmetric polynomials of $X$ could be computed, particularly the elementary symmetric polynomials, a univariate polynomial having all elements of $X$ as its roots could be constructed. 

\subsection{Constructing the polynomial we needed}
To apply Vieta's formulas to construct the polynomial we needed, it is necessary to compute the values of the elementary symmetric polynomials of $X$ using $X_{n,k}$.

\subsubsection{Computing $P_1$}
Consider $$S_1=\sum_{1\leq i_1<i_2<\cdots<i_k\leq n}(x_{i_1}+x_{i_2}+\cdots+x_{i_k})=\binom{n-1}{k-1}P_{1},$$
then the value of $P_1$ can be determined uniquely with $S_1$ given.

In addition, we have $P_1-X_{n,k}=X_{n,n-k}$ and if $k=1$, the multiset $X$ is not hidden. Hence, without losing generality, we assume $1<k\leq n/2.$

\subsubsection{Computing $P_2,P_{(1,1)}$}
Consider $$S_2=\sum_{1\leq i_1<i_2<\cdots<i_k\leq n}(x_{i_1}+x_{i_2}+\cdots+x_{i_k})^2=\binom{n-1}{k-1}P_{2}
+ \binom{n-2}{k-2}P_{(1,1)},$$
and $P_1^2=P_{2}+P_{(1,1)}$, then we have
$$\begin{pmatrix} \binom{n-1}{k-1} & \binom{n-2}{k-2}\\
1&1
\end{pmatrix}
\begin{pmatrix}P_2 \\P_{(1,1)}
\end{pmatrix}=
\begin{pmatrix}
S_2 \\P_1^2
\end{pmatrix} .$$
Since the matrix $\begin{pmatrix} \binom{n-1}{k-1} & \binom{n-2}{k-2}\\
1&1
\end{pmatrix}$ is invertible, the values of $P_2$ and $ P_{(1,1)}$ can be determined uniquely with $S_2$ and $P_1$ given.

\subsubsection{Computing $P_{p}$ for any partition $p$ of $u$}

Generally, consider \small{$$\begin{aligned} &S_{u}=\sum_{1\leq i_1<i_2<\cdots<i_k\leq n}(x_{i_1}+x_{i_2}+\cdots+x_{i_k})^u
=\binom{n-1}{k-1}P_{u}+
\\
& \binom{n-2}{k-2}
\left(\frac{\binom{u}{1}P_{(1,u-1)}}{\mu(1,u-1)}
+\frac{\binom{u}{2}P_{(2,u-2)}}{\mu(2,u-2)}+\cdots
+\frac{\binom{u}{\lfloor \frac{u}{2}\rfloor}P_{(\lfloor \frac{u}{2}\rfloor,u-\lfloor \frac{u}{2}\rfloor)}}{\mu(\lfloor \frac{u}{2}\rfloor,u-\lfloor \frac{u}{2}\rfloor)}\right)+
\\ &
\binom{n-3}{k-3}
\left(
\frac{\binom{u}{1}\binom{u-1}{1}P_{(1,1,u-2)}}{\mu(1,1,u-2)}+
\frac{\binom{u}{1}\binom{u-1}{2}P_{(1,2,u-3)}}{\mu(1,2,u-3)}+
\cdots+ 
\frac{\binom{u}{\lfloor \frac{u}{3}\rfloor}\binom{u-\lfloor \frac{u}{3} \rfloor}{\lfloor\frac{u-\lfloor \frac{u}{3} \rfloor}{2}\rfloor}P_{(\lfloor \frac{u}{3} \rfloor,\lfloor \frac{ u-\lfloor \frac{u}{3} \rfloor}{2}\rfloor,u-\lfloor \frac{u}{3} \rfloor-\lfloor\frac{u-\lfloor \frac{u}{3} \rfloor}{2}\rfloor}}{\mu(\lfloor \frac{u}{3} \rfloor,\lfloor\frac{u-\lfloor \frac{u}{3} \rfloor}{2}\rfloor,u-\lfloor \frac{u}{3} \rfloor-\lfloor\frac{u-\lfloor \frac{u}{3} \rfloor}{2}\rfloor}
\right)
\\ &
+\cdots+
\\ &
\binom{n-u}{k-u}\frac{\binom{u}{1}\binom{u-1}{1}\cdots\binom{2}{1}P_{(1,1,\dots,1)}}{\mu(1,1,\dots,1)} 
\end{aligned}.$$}

Notice that we assume $\binom{m_1}{m_2}=0,$ if $m_1<0$ or $m_2<0$ in this article.

Consider
$$P_1P_{u-1},P_2P_{u-2},\dots,P_{\lfloor \frac{u}{2}\rfloor}P_{u-\lfloor \frac{u}{2} \rfloor},$$
$$P_1P_1P_{u-2},P_1P_2P_{u-3},\dots,P_1^u,$$
where they satisfy
$$P_{j_1}P_{j_2}\cdots P_{j_{\ell}}= P_{j_1+\cdots+j_{\ell}}+
\sum_{s\in \{j_1,\dots,j_{\ell}\}}P_{(s,\sum_{t\in \{j_1,\dots,j_{\ell}\}-\{s\}})}+\cdots+P_{(j_1,\dots,j_{\ell})} $$
or
$$P_{j_1}P_{j_2}\cdots P_{j_{\ell}}=\sum_{\text{all partitions } p \text{ of } u}v((j_1,j_2,\dots,j_{\ell}),p)P_{p}$$
where $v(p_1,p_2)$ denotes the coefficient of $P_{p_2}$ in $\prod_{j_r\in p_1}P_{j_r}$ for any two partitions $p_1,p_2$ of $u$. Specifically, $v(p_1,p_2)$ is equal to the number of ways to combine $p_1$ into $p_2$ with $v(p,p)=1,$ for any partition $p$. For example, $v((1,1,2),(1,3))=2,$ $v((1,1,1),(1,2))=3,$ and $v((1,1,1,1),(2,2))=3,v((1,1,1,1),(1,1,2))=6.$

Hence, the linear transformation matrix from $(
    P_u, P_{(1,u-1)} , \dots , P_{(1,1,\dots,1)})^T$ to $(S_u , P_1P_{u-1} , \dots, (P_1)^u)^T$ is given by
\begin{align}\label{eq:mat}
\begin{pmatrix}
\binom{n-1}{k-1}& \binom{n-2}{k-2}
\frac{\binom{u}{1}}
{\mu(1,u-1)}&
\cdots &
\binom{n-2}{k-2}\frac{\binom{u}{\lfloor \frac{u}{2}\rfloor}}{\mu(\lfloor \frac{u}{2}\rfloor,u-\lfloor \frac{u}{2}\rfloor)}&
\binom{n-3}{k-3}\frac{\binom{u}{1}\binom{u-1}{1}}{\mu(1,1,u-2)}&
\cdots & \binom{n-u}{k-u}\\
1 & 1 & \cdots   &0&0 &\cdots & 0\\
\vdots & \vdots & \ddots  &\vdots &\vdots&\ddots &\vdots\\
1 & 0  & \cdots &1&0&\cdots &0\\
1 & \nu\left(\makecell{(1,1,u-2),\\(1,u-1)}\right) & \cdots & \nu\left(\makecell{(1,1,u-2),\\(\lfloor \frac{u}{2}\rfloor,u-\lfloor \frac{u}{2}\rfloor)}\right) & 1&\cdots & 0\\
\vdots & \vdots & \ddots  &\vdots &\vdots&\ddots &\vdots\\
1& \nu\left(\makecell{(1,1,\dots,1),\\(1,u-1)}\right) & \cdots & \nu\left(\makecell{(1,1,\dots,1),\\ (\lfloor \frac{u}{2}\rfloor,u-\lfloor \frac{u}{2}\rfloor)}\right) & \nu\left(\makecell{(1,1,\dots,1),\\(1,1,u-2)}\right) &\cdots & 1\\
\end{pmatrix},
\end{align} i.e. the matrix equation is 
$$\text{Matrix \eqref{eq:mat}}\begin{pmatrix}
    P_u \\ P_{(1,u-1)} \\ \vdots \\ P_{(1,1,\dots,1)}
\end{pmatrix}=\begin{pmatrix}
    S_u \\P_1P_{u-1}\\ \vdots \\ (P_1)^u
\end{pmatrix}$$
If Matrix \eqref{eq:mat} is invertible, the values of $P_{p}$ for all partitions $p$ of $u$ can be determined uniquely with $S_u,P_1,P_2,\dots,P_{u-1}$ given. In Matrix \eqref{eq:mat}, each row and each column correspond to a partition, respectively. For example, the $2$nd row corresponds to the partition 
$(1,u)$, and the $1$st column corresponds to the partition $(u)$.

Since $\binom{n-u}{k-u}=0$ when $u > k$, in this case, Matrix \eqref{eq:mat} can be replaced by its sub-matrix, which retains only the rows and columns corresponding to the partitions $p$ of $u$ into at most $k$ parts, allowing us to obtain the corresponding $P_u$.

\subsubsection{Applying Newton's identities and Vieta's formulas}
If the linear transformation matrices (i.e., Matrix \eqref{eq:mat}) for $u = 1, 2, \dots, n$ are all invertible, then the corresponding $P_{(1,1,\dots,1)}$'s can be determined uniquely. Consequently, the elementary symmetric polynomials $e_i$'s of $x_1, x_2, \dots, x_n$ can be determined as follows: let $e_1 = P_1$, and $e_i = \frac{P_{(1,1,\dots,1)}}{i!}$ for $i = 2, 3, \dots, k$, where $(1,1,\dots,1)$ consists of $i$ ones. By considering $\binom{n-u}{k-u} = 0$ when $u > k$, the corresponding $P_u$ can be obtained, and we can reduce computation using Newton's identities:
$$
ie_{i} = \sum_{j=1}^{i} (-1)^{j-1} e_{i-j} P_{j},
$$
where $e_0 = 1$.

Then the polynomial $$f(x):=\prod_{i=1}^n(x-x_i)=x^n+\sum_{i=n-1}^{0}(-1)^{n-i}e_{n-i}x^i$$ is what we needed by Vieta's formulas. \textbf{Hence, if the values of $e_i,i=1,2,\dots,n$ could be determined uniquely, then the multiset $X$ is recovered by finding all roots of the polynomial $f(x)$.}

\subsection{The determinant of Matrix \eqref{eq:mat}}
Before presenting the algorithm, we first compute the determinant of Matrix \eqref{eq:mat}. If Matrix \eqref{eq:mat} is not invertible, the algorithm fails. Therefore, we derive the determinant of the matrix and analyze the conditions under which the algorithm becomes invalid.

\begin{theorem}\label{thm:det}
The determinant of Matrix \eqref{eq:mat} is
$$\sum_{i=1}^{u}(-1)^{i-1}(i -1)!S(u,i)\binom{n-i}{k-i}, \label{eq:det}$$ where $S(u,i)$ is the Stirling partition number.
\end{theorem}

\begin{proof}
First, consider the Laplace expansion along the first row of Matrix \eqref{eq:mat}. Let $l_i$ be the number of partitions with $i$ elements. The determinant is then given by the following expression: $$\sum_{i=1,\ldots,u,j=1,\ldots,l_i}d_{i,j}\frac{\prod_{t=1}^{i-1}\binom{u_{i,j}[t]}{p_{i,j}[t]}}{\mu(p_{i,j})} \binom{n-i}{k-i},$$
 where $p_{i,j}$ is the $j$-th partition into $i$ parts, $d_{i,j}$ is the algebraic cofactor of the item corresponding to $p_{i,j}$, and $u_{i,j}[1]=u,u_{i,j}[t]=u_{i,j}[t-1]-p_{i,j}[t-1]$ for $t=2,3,\dots,i-1$. From the structure of Matrix \eqref{eq:mat}, we observe that $d_{1,1}=1.$

 Matrix \eqref{eq:mat} is non-invertible if and only if there exists non-trivial $$c_{1,1},c_{2,1},\ldots,c_{2,l_2},c_{3,1},\ldots,c_{3,l_3},\ldots, c_{u,l_u}$$ satisfying
\begin{align*}
    \sum_{i=1,\ldots,u,j=1,\ldots,l_i} c_{i,j}\bm v_{i,j}=\bm 0
\end{align*}
where $\bm v_{i,j}$ denotes the $\big((\sum_{t=1}^{i-1} l_t)+j\big)$-th column of Matrix \eqref{eq:mat}, $\bm 0$ denotes zero vector of appropriate size.

Let $\hat{\bm v}_{i,j}$ be the vector of removing the first entry of $\bm v_{i,j}$, the above equality also holds for all $\hat{\bm v}_{i,j}$s, i.e.,
\begin{align*}
    \sum_{i=1,\ldots,u,j=1,\ldots,l_i} c_{i,j}\hat{\bm v}_{i,j}=\bm 0.
\end{align*}
Considering only the first $1+l_2$ rows of the above equation we have that
\begin{align*}
   \forall j=\{1,\ldots,l_2\}: c_{1,1}+c_{2,j}=0.
\end{align*}
Thus all $c_{2,j}$'s are identical and equal to $-c_{1,1}$. Follow similarly as above, we have
\begin{align*}
    \forall j=\{1,\ldots,l_3\}: c_{1,1}+S(3,2)c_{2,j}+c_{3,j}=0,
\end{align*}
thus all $c_{3,j}$'s are identical. Repeat the above process, we conclude that for all $i={2,\ldots,u}$, all corresponding $\{c_{i,j}\}_{j=1,\ldots,l_i}$ are identical. For notatin simplicity we omit the $j$ notation and remark that all constants $c$ should satisfy the following
\begin{align}\label{eq6}
\begin{pmatrix}
 1&1&0&0&0&\ldots&0\\
 1&S(3,2)&1&0&0&\ldots&0\\
 1&S(4,2)&S(4,3)&1&0&\ldots&0 \\
 1&S(5,2)&S(5,3) &S(5,4) &1&\ldots&0\\
 \vdots &\vdots &\vdots &\vdots &\vdots&\ddots&\vdots\\
  1 &S(u,2) &S(u,3) &S(u,4) &S(u,5)&\ldots &1
\end{pmatrix}
\begin{pmatrix}
c_1\\ c_2\\ c_3\\ c_4 \\ \vdots\\c_u
\end{pmatrix} =
\begin{pmatrix}
0\\ 0\\ 0\\0\\ \vdots\\0
\end{pmatrix}
.\end{align}

Without losing generality, set $c_1=1$. Then we claim $c_i=(-1)^{i-1}(i-1)!$ for all $i$. Next, we will use the recursive method to prove this claim:

If $i=2$, since $c_1+c_2=0$, we have $c_2=-1=(-1)^11!$. If $i=3$, since $c_1+c_2S(3,2)+c_3=0$, we have $c_3=-c_1-c_2S(3,2)=-1-(-1)\cdot3=2=(-1)^22!.$
Assume $c_i=(-1)^{i-1}(i-1)!$ holds for $1\leq i < j$, then we consider $c_{j}$.
From Equality (\ref{eq6}), we obtain
$$
c_j=-\sum_{k=1}^{j-1}c_kS(j,k)
=(-1)^{j-1}(j-1)!S(j-1,j-1)
=(-1)^{j-1}(j-1)!,$$
where the second equality holds by the alternating sum formula $\sum_{k=1}^{j}(-1)^{j-1}(j-1)!S(j,k)=0$. Thus, the claim holds, i.e. $c_i=(-1)^{i-1}(i-1)!$ for all $i$. 

Since $d_{1,1}=1=c_{1,1}$, it follows that $d_{i,j}=c_i=(-1)^{i-1}(i-1)!$ for all $j$ by the following Lemma \ref{matrixlem}. Thus, the determinant of Matrix \eqref{eq:mat} is 
\begin{equation*}
    \sum_{i=1,\ldots,u,\\j=1,\ldots,l_i}(-1)^{i-1}(i-1)!\frac{\prod_{t=1}^{i-1}\binom{u_{i,j}[t]}{p_{i,j}[t]}}{\mu(p_{i,j})} \binom{n-i}{k-i}=\sum_{i=1}^{u}(-1)^{i-1}(i-1)!S(u,i)\binom{n-i}{k-i},
\end{equation*}
where $S(u,i)=\sum_{j=1,\ldots,l_i}\frac{\prod_{t=1}^{i-1}\binom{u_{i,j}[t]}{p_{i,j}[t]}}{\mu(p_{i,j})}$ by definition.
\end{proof}

\begin{lemma} \label{matrixlem}
For any square matrix $\mathbf{M}=(m_{ij})$ of order $m$, its determinant could be given as $\sum_{i=1}^{m} A_{1i}m_{1i}$, where $A_{1i}$ is the algebraic cofactor of $m_{1i}$. Consider the system of linear equations $\mathbf{M}'(y_2,y_3,\dots,y_m)=(-m_{21},-m_{31},\dots,-m_{m1})$, where $\mathbf{M}'$ is the invertible sub-matrix of $\mathbf{M}$ removing the first row and the first column, then $y_i=\frac{A_{1i}}{A_{11}},i=2,3,\dots,m$.
\end{lemma}
\begin{proof}
    It follows from Cramer rule that $y_i=\frac{\det(\mathbf{M}_i')}{\det(\mathbf{M}')}$ where $\mathbf{M}_i'$ is the matrix replacing the $(i-1)$-th column of the matrix $\mathbf{M}'$ with $(-m_{21},-m_{31},\dots,-m_{m1})$ for all $i=2,3,\dots,m$. Since $\det(\mathbf{M}')=(-1)^{1+1}A_{11}=A_{11}$ and $\det(\mathbf{M}_i')=-(-1)^{i-2}(-1)^{i+1}A_{1i}=A_{1i}$ for all $i=2,3,\dots,m$. The lemma follows.
\end{proof}

In fact, the determinant of Matrix (\ref{eq:mat}) is equal to the Moser polynomial\cite{selfridge1958determination,fomin2019multiset}: $$\sum_{i=1}^{k}(-1)^{i-1}i^{u-1}\binom{n}{k-i},$$ which can be verified using the identity $\sum_{j=i}^{k}\binom{j-1}{i-1}\binom{n-j}{k-j}=\binom{n}{k-i}$\cite{ewell1968determination}. 
\begin{remark}
    For any positive integers $k,u\leq n$, the following equality holds:
$$\sum_{i=1}^{u}(-1)^{i-1}(i -1)!S(u,i)\binom{n-i}{k-i}=\sum_{j=1}^{k}(-1)^{j-1}j^{u-1}\binom{n}{k-j}.$$
\end{remark}

\begin{proof}
\begin{align*}
    &\sum_{i=1}^{u}(-1)^{i-1}(i -1)!S(u,i)\binom{n-i}{k-i}\\
    =&\sum_{i=1}^{k}(-1)^{i-1}(i-1)!\frac{1}{i!}\sum_{j=0}^{i}\left((-1)^{i-j}\binom{i}{j}j^u\right)\binom{n-i}{k-i}\\
    =&\sum_{j=1}^{k}(-1)^{j+1}j^{u-1}\sum_{i=j}^{k}\binom{i-1}{j-1}\binom{n-i}{k-i}\\
    =&\sum_{j=1}^k(-1)^{j-1}j^{u-1}\binom{n}{k-j}
\end{align*}
\end{proof}

\begin{remark}
     For fixed $k$ and $u$, the Morse polynomial can be viewed as a $k-1$-degree polynomial in $n$, with a leading coefficient of $\frac{1}{(k-1)!}$ and a constant term of $(-1)^{k-1}k^{u-1}$. 
\end{remark}
\begin{proof}
    The proof is straightforward and thus omitted.
\end{proof}

Hence, when $u$ and $k$ are fixed, the values of $n$ for which Matrix \eqref{eq:mat} is not invertible must satisfy $n|(k-1)!k^{u-1}$.

\vspace{1em}\noindent\textbf{Some examples}

Fixing $k=2$ and any $u$, the determinant of Matrix \eqref{eq:mat} is $$(n-1)-(2^{u-1}-1)=n-2^{u-1}.$$  For this case, we know that only when $n=2^{u-1}$ the matrix is not invertible.

If $k=3$, the determinant is $$\begin{aligned}\binom{n-1}{2}-\binom{n-2}{1}(2^{u-1}-1)
+(3^{u-1}-2^u+1)=\frac{1}{2}\cdot\left(n^2-(2^u+1)n+2\cdot 3^{u-1}\right)
\end{aligned}.$$

If $k=4,$ the determinant is
$$\frac{1}{6}\left(n^3-3(2^u+1)n^2+(2\cdot3^u+3\cdot2^{u-1}+2)n-(6\cdot4^{u-1})\right).$$

\subsection{Symmetric subset sum recovery algorithm}

\textit{Proof of Theorem \ref{thm:vie}}

Incorporating the content mentioned above in this section, we have the following algorithm, which is called symmetric subset sum recovery algorithm, for convenience.
\begin{breakablealgorithm}
    \renewcommand{\algorithmicrequire}{\textbf{Input:}}
	\renewcommand{\algorithmicensure}{\textbf{Output:}}
	\caption{Symmetric subset sum recovery algorithm}
    \label{alg:vien}
    \begin{algorithmic}[1] 
        \REQUIRE Positive integers $k,n$ satisfying $2\leq k\leq n/2$ and the multiset $X_{n,k}=\{b_1,b_2,\dots,b_m\}$ where $m=\binom{n}{k}$; 
	    \ENSURE The hidden multiset $X=\{x_1,x_2,\dots,x_n\}$; 

         \IF{$n|(k-1)!k^n$}
            \FOR{$u=1,2,...,n$}
              \IF{$\sum_{i=1}^{k}(-1)^{i-1}i^{u-1}\binom{n}{k-i}$}
                 \RETURN The algorithm fails in this case!
             \ENDIF
             \ENDFOR
         \ENDIF

         \STATE Set $\mathbf{P}=[]$ and $\mathbf{E}=[1]$; 
        \FORALL{$u=1,2,\dots,n$}

            \STATE  Let the array $\mathbf{L}$ consist of all partitions of $u$ into at most $k$ parts, arranged in ascending order by the cardinality of the partitions. Specifically,  $\mathbf{L}[1] = (u)$, $\mathbf{L}[n(\mathbf{L})] = (1, 1, \dots, 1)$, where $n(\mathbf{L})$ denotes the length of the array $\mathbf{L}$;
            
           \STATE Initialize a square zero matrix $\mathbf{M}$ of order $n(\mathbf{L})$;

          \FORALL{$j=1,2,\dots,n(\mathbf{L})$}
              \IF{$k > n(\mathbf{L}[j])$}
                \STATE break;
              \ENDIF
            \STATE Let $\mathbf{M}(1,j)=\binom{n-n(\mathbf{L}[j])}{k-n(\mathbf{L}[j])}\frac{\prod_{l=1}^{n(\mathbf{L}[j])-1}\binom{u_{l}}{\mathbf{L}[j][l]}}{\mu(\mathbf{L}[j])}$ with $u_1=u,u_l=u_{l-1}-\mathbf{L}[j][l-1]$ for $l=2,3,\dots,n(\mathbf{L}[j])-1$ ;
           \ENDFOR
          \FORALL{$i=2,3,\dots,n(\mathbf{L})$}
             \FORALL{$j=1,2,\dots,n(\mathbf{L})$}
               \STATE Let $\mathbf{M}(i, j) = v(\mathbf{L}[i], \mathbf{L}[j])$;
            \ENDFOR
        \ENDFOR
           
        \STATE Let $\mathbf{S}=[\sum_{i=1}^{m}(B[i])^u]$;
        \FORALL{$i=2,3,\dots,n(\mathbf{L})$}
           \STATE Let $\mathbf{S}=[op(\mathbf{S}),\prod_{j=1}^{n(\mathbf{L}[i])}\mathbf{P}_{\mathbf{L}[i][j]}]$
        \ENDFOR
          
          \STATE Let $\mathbf{P}=[op(\mathbf{P}),(\mathbf{M}^{-1}\mathbf{S})[1]]$
    \ENDFOR

    \FORALL{$i=1,2,\dots,n$}
       \STATE Let $\mathbf{E}=[op(\mathbf{E}),\frac{1}{i}\left(\sum _{j=1}^{i}(-1)^{j-1}\mathbf{E}[i-j+1]\mathbf{P}[j]\right)]$;
    \ENDFOR

   \STATE Find the $n$ roots of the polynomial $x^n+\sum_{i=n-1}^{0}(-1)^{n-i}\mathbf{E}[n-i+1]x^i$ of degree $n$ in one variable and let these roots form the multiset $X$;

   \RETURN $X$
    
    \end{algorithmic}
\end{breakablealgorithm}

\noindent\textbf{Complexity analysis}

The time complexity of computing all $S_u$'s is $O\left(\binom{n}{k}n\right)$. The time complexity of generating all partitions into at most $k$ parts and computing all $P_{i_1} P_{i_2}  \cdots P_{i_k}$'s is $$O\left(\sum_{u=1}^{n}p(u,\leq k) \min(u,k)\right).$$  To facilitate efficient calculation of combinations, the first $n$ rows of Pascal's triangle can be precomputed with a time complexity of $O(n^2)$ \cite{leiserson1994introduction}. The overall time complexity for computing all $v(p_1,p_2)$ values is $O(\sum_{u=1}^{n}p(u,\leq k)^2\min(u,k)^2)$ using a dynamic programming approach. Consequently, the time complexity of constructing all Matrix \eqref{eq:mat}'s can be expressed as $$O\left(n^2+\sum_{u=1}^np(u,\leq k) ^2\min(u,k)^2+p(u,\leq k) \min(u,k)\right).$$ Computing the inverse of Matrix \eqref{eq:mat} or solving the associated system of linear equations  takes $O\left(\sum_{u=1}^n p(u,\leq k)^3\right)$ time \cite{trefethen1997numerical}. Using the characteristic polynomial method, the time complexity of finding all roots of a univariate polynomial of degree $n$ is $O(n^3)$ \cite{edelman1995polynomial,trefethen1997numerical}. 
Thus, the total time complexity, ignoring some lower-complexity computations, becomes
$$O\left(\sum_{u=1}^n \left(p(u,\leq k) ^3+(p(u,\leq k) \min(u,k)+1)^2\right)+\binom{n}{k}n+n^2+n^3\right),$$
which, by ignoring lower-order terms further, is asymptotically equivalent to $$O\left(\sum_{u=1}^n p(u,\leq k) ^3+\binom{n}{k}n\right).$$ 
This completes the proof of Theorem \ref{thm:vie}. $\hfill\qedsymbol$

 Algorithm \ref{alg:vien} can be generalized to the \((n, k)\)-complete HSSP defined over other fields or rings, where $n!$ and every nonzero value of  
$$\sum_{j=1}^{k}(-1)^{j-1}j^{u-1}\binom{n}{k-j}, \quad u = 1, 2, \dots, n,$$ 
are invertible and Vieta's formulas hold.

As an additional note, the generating function \cite{andrews2004integer} of $p(u)$ is given by
$$\sum_{u=0}^{\infty}p(u)x^u=\prod_{j=1}^{\infty}\sum_{i=0}^{\infty}x^{ji}
=\prod_{j=1}^{\infty}(1-x^j)^{-1} .$$
Currently, no closed-form expression for the partition function is known, but it has both asymptotic expansions that accurately approximate it and recurrence relations by which it can be calculated exactly.  As $u$ becomes large, $p(u)$ grows exponentially in the square root of its argument\cite{andrews2004integer,hardy1918asymptotic}. Specifically, it has the asymptotic form:
$$ p(u)\backsim \frac{1}{4u\sqrt{3}}\exp\left(\pi\sqrt{\frac{2u}{3}}\right) \text{ as } u\rightarrow \infty.$$

\subsection{Nontrivial permutation-symmetric subset of $X_{n,k}$}
This method appears to be applicable to any permutation-symmetric subset of $X_{n,k}$, where permutation-symmetric means the subset of $X_{n,k}$ remains invariant under any permutation of $x_1,x_2,\dots,x_n$. However, we claim that $X_{n,k}$ does not contain a nontrivial permutation-symmetric subset in which every element of $X$ contributes to the subset sums. If some elements of $X$ do not contribute to the subset sums of a subset of $X_{n,k}$, the subset might exhibit symmetry. However, we do not consider such cases, as they can be effectively treated as instances with a smaller $n$.

\begin{theorem}
Let $X_{n,k}$ be the multiset of all $k$-subset sums of finite multiset $X=\{x_1,x_2,\dots,x_n\}$ of cardinality $n$, then $X_{n,k}$ does not have nontrivial permutation-symmetric subset with every element of $X$ contributing to the subset sums.
\end{theorem}
\begin{proof}
Suppose $B$ is such a permutation-symmetric subset of $X_{n,k}$. Since $B$ is not empty, there exists $x_{i_1}+x_{i_2}+\cdots+x_{i_k} \in B$. And since $B\ne X_{n,k}$, there exists $x_{i_1'}+x_{i_2'}+\cdots+x_{i_k'} \notin B$. By the symmetry property, we could arbitrarily exchange the positions of any two elements of $X$. For any $j=1,2,\dots,k$, if $\{i_1,\dots,i_k\} \cap \{i_1',\dots,i_k'\}=\emptyset$, exchange $x_{i_j}$ and $x_{i_j'}$, then $x_{i_1'}+x_{i_2'}+\cdots+x_{i_k'} \in B$, which is a contradiction. If $\{i_1,\dots,i_k\} \cap \{i_1',\dots,i_k'\}$ is not empty, keep the indexes in the intersection set unchange, and exchange the other indexes. Then we could still obtain a contradiction.
\end{proof}

\subsection{The homogeneous symmetric polynomial rings }
Fixing $n$ and $k$, we compute $P_u$ recursively from $u = 1$ to $n$ as described above. However, the Matrix \eqref{eq:mat} may be non-invertible for certain values of $u$. In such cases, it is not possible to uniquely determine the value for any one partition of $u$ without additional information in Algorithm \ref{alg:vien}, leading to the failure of the algorithms. This raises the following question: 
\begin{question}\label{ques1}
    When the Matrix \eqref{eq:mat} becomes non-invertible for certain values of $u$, can we still uniquely determine the values for the partitions of $v$ where $v \geq u$ by modifying Algorithm \ref{alg:vien} without additional information?
\end{question}

Before answering this question, let us first consider the homogeneous symmetric polynomial ring of degree $u$, denoted here by $R_u$, and prove the following theorem.
\begin{theorem}\label{thm:hom}
For a positive integer of $u>1$ and the homogeneous symmetric polynomial ring $R_u$ of degree $u$ over $\mathbb{R}$. A basis of $R_u$ can be constructed by $S_u$ and the products of the elements of $R_v,v<u$ if and only if  Matrix \eqref{eq:mat} is invertible. This is is equivalent to $$\sum_{j=1}^{k}(-1)^{j-1}j^{u-1}\binom{n}{k-j}\ne 0.$$
\end{theorem}
\begin{proof}
    The ring $R_u$ can be taken as a linear space over $\mathbb{R}$, whose dimension is equal to the number of partitions of $u$. Matrix \eqref{eq:mat} is taken as the transformation matrix from the vector $(P_u,P_{(1,u-1)},\dots, P_{(1,1,\dots,1)})$ to the vector $(S_u,P_1P_{u-1},\dots, P_1^u)$. The former vector always forms a basis of $R_u$. For the latter vector, except the first component, the other components are computed by the homogeneous symmetric polynomials of lower degree. From the structure of Matrix \eqref{eq:mat}, it could be observed that $\{P_1P_{u-1},\dots, P_1^u\}$ are linearly independent. At this time, we could not obtain any new item, that is linearly independent of $\{P_1P_{u-1},\dots, P_1^u\}$, from $R_v, v<u$. And $S_u$ is linearly independent from $\{P_1P_{u-1},\dots, P_1^u\}$ if and only if Matrix \eqref{eq:mat} is invertible. Thus, $\{S_u,P_1P_{u-1},\dots, P_1^u\}$ form a basis of $R_u$ if and only if Matrix \eqref{eq:mat} is invertible. Proof of this theorem is thus complete.
\end{proof}

Back to Question \ref{ques1}, the primary available data is the set $X_{n,k}$, which contains all $k$-subset sums of $n$ elements. Consequently, $S_u$ can be uniquely determined for any $u \in \mathbb{Z}_{>0}$. Then, starting from $u=1$ to larger values, we can get the values $P_{p}$ of all partitions. Therefore, we can assume $R_v$'s are known for all $v<u$ before computing $R_u$. The proof of Theorem \ref{thm:hom} shows that choosing $\{P_1P_{u-1},\dots, P_1^u\}$ is the same as choosing other products of the elements of $R_v,v<u$ to form a basis. Hence, $\{S_u,P_1P_{u-1},\dots, P_1^u\}$ is chosen to compute $P_{p}$, where $p$ is a partition of $u$. If we can not determine the value of any $P_{p}$ of degree $u$ uniquely with this choice, neither could the others.

Furthermore, if there exists $1<u_0<n$ such that Matrix \eqref{eq:mat} is non-invertible, then a basis of $R_{u_0}$ cannot be constructed from $S_u$ for all $u\leq u_0$. At this time, can a basis of $R_{u}$, where $u\geq u_0$, be constructed from  $S_v$ for all $v \leq u$? To answer this question, we have the following theorem.

\begin{theorem}\label{thm:homl}
    If there exists $1<u_0<n$ such that Matrix \eqref{eq:mat} is non-invertible when $u=u_0$ and is invertible for all $u<u_0$, then a basis of $R_{u}$, where $u<u_0$, can be constructed from $S_v,v<u_0$. But a basis of $R_{u}$, where $u\geq u_0$, cannot be constructed from $S_v,v\leq u$.
\end{theorem}
\begin{proof}
    The conclusion for the cases where $u\leq u_0$ follows from Theorem \ref{thm:hom}. We prove the conclusion for the cases $u > u_0$ using a recursive method.
    When $u=u_0+1$, now we aim to construct a basis of $R_{u_0+1}$, which consists of $p(u_0+1)$ elements. First, we select $p(u_0+1)-2$ linearly independent elements $$P_2P_{u_0-1},P_3P_{u_0-2},\dots, P_1^2P_{u_0-1},\dots,P_1^v$$ as the partial basis. At this time, for any $u_i$'s,$v_i$'s satisfying $u_i<u_0,v_i<v_0$ and $$u_1+\cdots +u_r+v_1+\cdots +v_s=u_0+1,$$ the element $P_{p_1}\cdots P_{p_r}S_{v_1}\cdots S_{v_s}$, where $p_i$ is a partition of $u_i$, is linearly dependent with $P_2P_{u_0-1},P_3P_{u_0-2},\dots, P_1^2P_{u_0-1},\dots,P_1^v$, since every $P_{p_i}$ or $S_{v_i}$ can be represented by a linear combination of $P_1P_{u_i-1},\dots, P_1^{u_i}$ or $P_1P_{v_i-1},\dots, P_1^{v_i}$.
    To find enough linearly dependent items, we must include $S_{u_0+1}$ and $P_1S_{u_0}$. However, $P_1S_{u_0}$ is still linearly dependent with $P_2P_{u_0-1},P_3P_{u_0-2},\dots, P_1^2P_{u_0-1},\dots,P_1^v$. Hence, we cannot construct a basis of $R_{u_0+1}$ from $S_v,v\leq u_0+1$.

   Assume the conclusion holds for $u_0\leq v< w$, consider $w$. A basis of $R_{w}$ consists of $p(w)$ elements. First, we select all elements of the form $P_{1}^{e_1}P_{2}^{e_2}\cdots P_{u_0-1}^{e_{u_0-1}},$ where ${e_1}+2e_2+\cdots +(u_0-1){e_{u_0-1}}=w, e_i \in \mathbb{Z}_{\geq 0}.$  This corresponds to the partition consisting of $e_1$ parts of $1$, $e_2$ parts of $2$, $\dots,$ and $e_{u_0-1}$ parts of $u_0-1$. Similar to the discussion of the case $u=u_0+1$, any $P_{p_1}\cdots P_{p_r}S_{v_1}\cdots S_{v_s},$ where $p_i$ is a partition of $u_i$, $u_i<u_0, v_i<u_0$ and $u_1+\cdots +u_r+v_1+\cdots +v_s=v$, is linearly dependent with $P_{1}^{e_1}P_{2}^{e_2}\cdots P_{u_0-1}^{e_{u_0-1}}$'s. Then, $S_{u_0},S_{u_0+1},\dots,S_{w}$ must be used to construct the basis. Specially, to offer enough items, let $w=gu_0+ t, 0\leq t< u_0 $, the element $P_{t}S_{u_0}^g$, which corresponds to the partition consisting of one part of $t$, $g$ parts of $u_0$,  must be included to construct the basis. However, it is linearly dependent with $P_{1}^{e_1}P_{2}^{e_2}\cdots P_{u_0-1}^{e_{u_0-1}}$'s. Thus, a basis of $R_{w}$ cannot be constructed from  $S_v,v\leq u$.
\end{proof}

 Hence, by Theorem \ref{thm:homl}, we cannot get enough linearly independent items and then cannot determine any $P_{p}$, where $p$ is a partition of $v$ for any $v\geq u_0$, uniquely. This means the answer to Question \ref{ques1} is \textbf{No} with high probability!

The above discussion is based solely on the information of the set $X_{n,k}$. In actual scenarios, there might be supplementary data at hand, such as the range of elements. Since the rank of Matrix \eqref{eq:mat} is at least $n-1$, when Matrix \eqref{eq:mat} is not invertible, the possible value of $P_{p}$ is in the form of $\alpha+k_0\beta, k_0\in \mathbb{R}$, where $\alpha$  and $\beta$ can be determined, hence we may get the value of $P_{p}$ with its range known.


\section{Conclusion and future work} \label{sec:fur}
\noindent\textbf{Conclusion}
In this paper, we present two deterministic algorithms (Algorithm \ref{alg:bru} and Algorithm \ref{alg:vien}) for solving the $(n,k)$-complete HSSP. A comparison of the algorithms for solving HSSP is provided as follows:

\begin{table}[htbp]
\centering
\footnotesize
\caption{Comparison of algorithms for solving HSSP}
\label{tab:algorithm_comparison}
\begin{tabular}{|c|c|c|c|}
\hline
\textbf{Algorithms} & \textbf{\makecell{(best) Time \\ complexity}} & \textbf{\makecell{Deter-\\ministic?}} & \textbf{Requirements}  \\ \hline
\makecell{Original brute-force\\ search }& $O(\binom{\binom{n}{k}}{n}(n^3+\binom{n}{k}-n))$ & yes &  $(n,k)$-complete HSSP   \\ \hline
Our algorithm \ref{alg:bru} & \tiny{ $O\left(\binom{n}{k}\left(\log \binom{n}{k}+\prod_{i=1}^{k-1}\left(\binom{n-k+i}{i}-i\right)\right)\right)$} & yes & \makecell{$(n,k)$-complete HSSP \\with ordering relationship}  \\ \hline
 \makecell{Greedy heuristic \\ algorithms \cite{collins2007nonnegative, lozano2016genetic} } & $\backslash$ & no &  \makecell{HSSP whose elements\\ are positive integers}\\ \hline

\makecell{NS algorithms and\\ variations\cite{nguyen1999hardness,coron2020polynomial, coron2021provably,gini2022hardness}} & $O(n^9)$ & no  & \makecell{HSSP over $\mathbb{Z}_M$  satisfying\\ some lattice properties}\\ \hline

Our algorithm \ref{alg:vien} & $O\left(\sum_{u=1}^n p(u,\leq k) ^3+\binom{n}{k}n\right)$ & yes  & \makecell{$(n,k)$-complete HSSP\\ with unique solution} \\ \hline
\end{tabular}
\end{table}

For the $(n,k)$-complete HSSP, our algorithm \ref{alg:bru} demonstrates an advantage for small values of $k$, while NS algorithms and their variations perform well for $n=2k$. In other scenarios, our algorithm \ref{alg:vien} may provide competitive performance.

\vspace{1em}
\noindent\textbf{Future work}
There are some future directions we are interested in regarding HSSP: 1) Singular Pair Characterization: A deeper analysis of parameter pairs $(n,k)$ that admit multiple solutions, including the development of criteria to classify such pairs beyond the current necessary conditions.    
2) Applications in AI Privacy:  Optimizing these algorithms to analyze privacy risks in federated learning scenarios, where HSSP-based attacks reconstruct private data from aggregated updates.

\section*{Acknowledgements}
We would like to express our sincere gratitude to those who have supported and contributed to this paper. Special thanks go to Prof. Song Dai for his insightful discussions during the initial stages, which helped us to identify the right entry point for this paper. We are also deeply grateful to Dr. Agnese Gini for her valuable discussions and feedback throughout the process, which greatly improved the paper.
\bibliographystyle{unsrt} 
\bibliography{hsspbib}

\end{document}